\documentclass{amsart}
\usepackage{amscd,amssymb,amsxtra}

\theoremstyle{plain}
\newtheorem{Thm}{Theorem}[section]
\newtheorem{Prop}[Thm]{Proposition}
\newtheorem{Lem}[Thm]{Lemma}

\newtheorem*{PET}{The Pointwise Ergodic Theorem}

\theoremstyle{definition}
\newtheorem{Def}[Thm]{Definition}

\newtheorem{Notation}[Thm]{Notation}

\numberwithin{equation}{section}
\newcommand{\res}{\upharpoonright}

\newcommand{\reg}{\text{reg}}
\newcommand{\Alt}{\operatorname{Alt}}

\newcommand{\Fix}{\operatorname{Fix}}

\newcommand{\Var}{\operatorname{Var}}
\newcommand{\Wr}{\mathbin{\text{wr}}}

\newcommand{\Sub}{\text{Sub}}
\newcommand{\Sym}{\operatorname{Sym}}

\begin{document}


\begin{abstract}
We classify the ergodic invariant random subgroups of strictly diagonal limits
of finite symmetric groups.
\end{abstract}

\title
[Invariant random subgroups]{Invariant random subgroups of strictly diagonal limits of finite symmetric groups}


\author{Simon Thomas}
\address
{Mathematics Department \\
Rutgers University \\
110 Frelinghuysen Road \\
Piscataway \\
New Jersey 08854-8019 \\
USA}
\email{sthomas@math.rutgers.edu}
\author{Robin Tucker-Drob}
\email{rtuckerd@gmail.com}
\thanks{Research partially supported by NSF Grants DMS 1101597 and DMS 1303921.}


\maketitle

\section{Introduction} \label{S:intro}

Let $G$ be a countable discrete group and let $\Sub_{G}$ be the compact space
of subgroups $H \leqslant G$. Then a Borel probability measure $\nu$ on $\Sub_{G}$ which is invariant 
under the conjugation action of $G$ on $\Sub_{G}$ is called an {\em invariant random subgroup\/} or $IRS$. 
For example, if $N \trianglelefteq G$ is a normal
subgroup, then the Dirac measure $\delta_{N}$ is an IRS of $G$.
Further examples arise from from the stabilizer distributions of measure preserving actions, 
which are defined as follows. Suppose that $G$ acts via measure preserving maps on the 
Borel probability space $(\, Z, \mu \,)$ and let $f: Z \to \Sub_{G}$ be the
$G$-equivariant map defined by
\[
z \mapsto G_{z} = \{\, g \in G \mid g \cdot z = z \,\}.
\]
Then the corresponding {\em stabilizer distribution\/} $\nu = f_{*}\mu$ is an IRS 
of $G$. In fact, by a result of Abert-Glasner-Virag \cite{agv}, every IRS of $G$ can be realized as the 
stabilizer distribution of a suitably chosen measure preserving action. Moreover, 
by Creutz-Peterson \cite{cp}, if $\nu$ is an ergodic IRS of $G$, then $\nu$ is the stabilizer distribution 
of an ergodic action $G \curvearrowright (\, Z, \mu\,)$.

A number of recent papers have focused on the problem of studying the IRS's
of certain specific countably infinite groups. For example, Bowen \cite{b} has shown
that each free group $\mathbb{F}_{m}$ of rank $m \geq 2$ has a huge ``zoo'' of IRS's;
and Bowen-Grigorchuk-Kravchenko \cite{bgk} have proved that the same is
true of the lamplighter groups $(\, \mathbb{Z}/p\mathbb{Z})^{n} \Wr \mathbb{Z}$, where
$p$ is a prime and $n \geq 1$. On the other hand, Vershik \cite{v2} has given a complete
classification of the ergodic invariant random subgroups of the group of finitary 
permutations of the natural numbers.\footnote{There is a slight inaccuracy in the statement \cite{v2} of
Vershik's classification theorem.} In this paper, we will classify the ergodic invariant 
subgroups of the strictly diagonal limits of finite symmetric groups, which are defined as follows.

Suppose that $\Sym(\Delta)$, $\Sym(\Omega)$ are finite symmetric groups and that
$|\Omega| = \ell |\Delta|$. Then an embedding $\varphi: \Sym(\Delta) \to \Sym(\Omega)$ is said to be
an {\em $\ell$-fold diagonal embedding\/} if $\varphi(\Sym(\Delta))$ acts via its natural permutation
representation on each of its orbits in $\Omega$. The countable locally finite group $G$ is a 
{\em strictly diagonal limit\/} of finite symmetric groups if we can express 
$G = \bigcup_{n \in \mathbb{N}} G_{n}$ as the union of an increasing chain of finite
symmetric groups $G_{n} = \Sym(X_{n})$, where each embedding 
$G_{n} \hookrightarrow G_{n+1}$ is an $|X_{n+1}|/|X_{n}|$-fold diagonal embedding. 
In this case, we say that $G$ is an SD$\mathcal{S}$-group. Here, letting 
$[k] = \{\, 0,1, \cdots, k-1\,\}$, we can suppose that for some sequence $(k_{n})$ of
natural numbers $k_{n} \geq 2$, we have that $X_{n} = \prod_{0 \leq m \leq n} [k_{m}]$ and
that the embedding $\varphi_{n}: G_{n} \to G_{n+1}$ is defined by
\[
\varphi_{n}(g) \cdot (\, i_{0},\cdots, i_{n}, i_{n+1}\,) =
(\, g(i_{0},\cdots, i_{n}), i_{n+1}\,).
\]
Let $X = \prod_{n \geq 0} [k_{n}]$ and let $\mu$ be the product probability measure of the
uniform probability measures on the $[k_{n}]$. Then $G$ acts naturally on $(\, X, \mu\,)$ as
a group of measure preserving transformations via
\[
g \cdot (\, i_{0},\cdots, i_{n}, i_{n+1}, i_{n+2}, \cdots \,) =
(\, g(i_{0},\cdots, i_{n}), i_{n+1}, i_{n+2}, \cdots\,),  \quad g \in G_{n}.
\]
It is easily checked that the action $G \curvearrowright (\, X, \mu\,)$ is
weakly mixing and it follows that the diagonal action of $G$ on the product space
$(\, X^{r}, \mu^{\otimes r}\,)$ is ergodic for each $r \in \mathbb{N}^{+}$. Hence
the stabilizer distribution $\sigma_{r}$ of $G \curvearrowright (\, X^{r}, \mu^{\otimes r}\,)$
is an ergodic IRS of $G$. We will show that if $G$ is simple, then 
$\{\, \delta_{1}, \delta_{G} \,\} \cup \{\, \sigma_{r} \mid r \in \mathbb{N}^{+} \,\}$ is 
a complete list of the ergodic IRS's of $G$. A moment's thought shows that $G$ is simple if and only
if $k_{n}$ is even for infinitely many $n \in \mathbb{N}$. Suppose now that $k_{n}$ is odd
for all but finitely many $n \in \mathbb{N}$. Then clearly 
$A(G) = \bigcup_{n \in \mathbb{N}} \Alt(X_{n})$ is a simple subgroup of $G$ such that
$[\, G: A(G)\,] = 2$. For each $r \in \mathbb{N}^{+}$, let 
$\tilde{f}_{r}: X^{r} \to \Sub_{G}$ be the $G$-equivariant map defined 
\[
(x_{0}, \cdots, x_{r-1}) = \Bar{x} \mapsto G_{\Bar{x}} \cap A(G), 
\]
where $G_{\Bar{x}} = \{\, g \in G \mid g \cdot x_{i} = x_{i} \text{ for } 0 \leq i < r \,\}$.
Then $\tilde{\sigma}_{r} = (\tilde{f}_{r})_{*}\mu^{\otimes r}$ is also an ergodic IRS of $G$.

\begin{Thm} \label{T:main}
With the above notation, if $G$ is a simple SD$\mathcal{S}$-group, then the ergodic IRS's of $G$ are
\[
\{\, \delta_{1}, \delta_{G} \,\} \cup \{\, \sigma_{r} \mid r \in \mathbb{N}^{+} \,\};
\]
while if $G$ is a non-simple SD$\mathcal{S}$-group, then the ergodic IRS's of $G$ are
\[
\{\, \delta_{1}, \delta_{A(G)}, \delta_{G} \,\} \cup \{\, \sigma_{r} \mid r \in \mathbb{N}^{+} \,\}
\cup \{\, \tilde{\sigma}_{r} \mid r \in \mathbb{N}^{+} \,\}.
\]
\end{Thm}

By Creutz-Peterson \cite{cp}, in order to prove Theorem \ref{T:main}, it is enough to show that the
stabilizer distribution $\sigma$ of each ergodic action $G \curvearrowright (\, Z, \mu\,)$ is included
in the above list of invariant random subgroups. Our analysis of the action $G \curvearrowright (\, Z, \mu\,)$
will proceed via an application of the Pointwise Ergodic Theorem to the associated character
$\chi(g) = \mu(\, \Fix_{Z}(g)\,)$, which will enable us to regard $G \curvearrowright (\, Z, \mu\,)$
as the ``limit'' of a suitable sequence of finite permutation groups 
$G_{n} \curvearrowright (\, \Omega_{n}, \mu_{n} \,)$, where $\mu_{n}$ is the uniform probability measure 
on $\Omega_{n}$. (In other words, we will follow the {\em asymptotic approach to characters\/} of 
Kerov-Vershik \cite{vk,vk2}.)

\begin{Def} \label{D:char}
If $\Gamma$ is a countable discrete group, then the function $\chi: \Gamma \to \mathbb{C}$ is a
{\em character\/} if the following conditions are satisfied:
\begin{enumerate}
\item[(i)] $\chi(h\,g\,h^{-1}) = \chi(g)$ for all $g$, $ \in \Gamma$.
\item[(ii)] $\sum_{i,j=1}^{n} \lambda_{i} \Bar{\lambda}_{j} \chi(g_{j}^{-1}g_{i}) \geq 0$
for all $\lambda_{1}, \cdots, \lambda_{n} \in \mathbb{C}$ and $g_{1}, \cdots, g_{n} \in \Gamma$.
\item[(iii)] $\chi(1_{G}) = 1$.
\end{enumerate}
A character $\chi$ is said to be {\em indecomposable\/} or {\em extremal\/} if it is not possible to
express $\chi = r \chi_{1} + (1 - r) \chi_{2}$, where $0 < r < 1$ and $\chi_{1} \neq \chi_{2}$ are
distinct characters.
\end{Def}

In earlier work, Leinen-Puglisi \cite{lp} and Dudko-Medynets \cite{dm} classified the characters 
of the SD$\mathcal{S}$-groups; and combining their classification and Theorem \ref{T:main},
we obtain that if $G$ is a simple SD$\mathcal{S}$-group, then the indecomposable characters of $G$ are
precisely the associated characters of the ergodic IRS's of $G$. However, it should be stressed that our
work neither makes use of nor implies the classification theorem of Leinen-Puglisi and Dudko-Medynets.
(See Vershik \cite{v} for some fascinating conjectures concerning the relationship between
invariant random subgroups and characters.) 

This paper is organized as follows. In Section \ref{S:point}, we will discuss the Pointwise Ergodic Theorem for 
ergodic actions of countably infinite locally finite finite groups. In Section \ref{S:outline},
we will briefly outline the strategy of the proof of Theorem \ref{T:main}. In Sections
\ref{S:stirling} and \ref{S:prob}, we will prove a series of key lemmas concerning the asymptotic
values of the normalized permutation characters of various actions $S \curvearrowright S/H$,
where $S$ is a suitable finite symmetric group. Finally, in Section \ref{S:proof}, we will present
the proof of Theorem \ref{T:main}.

\section{The pointwise ergodic theorem} \label{S:point} 

In this section, we will discuss the Pointwise Ergodic Theorem for ergodic actions of countably
infinite locally finite finite groups. Throughout $G = \bigcup G_{n}$ is the union of a
strictly increasing chain of finite subgroups $G_{n}$ and $G \curvearrowright (\, Z, \mu\,)$ 
is an ergodic action on a Borel probability space. The following is a special case
of more general results of Vershik \cite[Theorem 1]{ver} and Lindenstrauss \cite[Theorem 1.3]{l}.

\begin{PET} \label{T:pet}
With the above hypotheses, if $f \in L^{1}(Z,\mu)$, then for $\mu$-a.e. $z \in Z$, 
\[
\int f\, d\mu = \lim_{n \to \infty} \frac{1}{|G_{n}|} \sum_{g \in G_{n}} f(g \cdot z).
\]
\end{PET}

In particular, the Pointwise Ergodic Theorem applies when $f$ is the characteristic function
of the Borel subset $\Fix_{Z}(g) = \{\, z \in Z \mid g \cdot z = z \,\}$ for some $g \in G$. From now on,
for each $z \in Z$ and $n \in \mathbb{N}$, let $\Omega_{n}(z) = \{\, g \cdot z \mid g \in G_{n}\,\}$
be the corresponding $G_{n}$-orbit. 

\begin{Thm} \label{T:point}
With the above hypotheses, for $\mu$-a.e. $z \in Z$, for all $g \in G$,
\[
\mu(\, \Fix_{Z}(g) \,) = \lim_{n \to \infty} |\, \Fix_{\Omega_{n}(z)}(g)\,|/|\, \Omega_{n}(z)\,|.  
\]
\end{Thm}

\begin{proof}
Fix some $g \in G$. Then by the Pointwise Ergodic Theorem, for $\mu$-a.e. $z \in Z$, 
\[
\mu(\, \Fix_{Z}(g) \,) = \lim_{n \to \infty} \frac{1}{|G_{n}|} 
|\, \{\, h \in G_{n} \mid h \cdot z \in \Fix_{Z}(g) \,\}\,|.
\]
Fix some such $z \in Z$; and for each $n \in \mathbb{N}$, let 
$H_{n} = \{\, h \in G_{n} \mid h \cdot z = z \,\}$ be the corresponding point stabilizer. Then clearly,
\[
|\, \{\, h \in G_{n} \mid h \cdot z \in \Fix_{Z}(g) \,\}\,| = 
|\, \Fix_{\Omega_{n}(z)}(g)\,|\,|\,H_{n}\,|;
\]
and so we have that
\begin{align*}
\frac{1}{|G_{n}|} |\, \{\, h \in G_{n} \mid h \cdot z \in \Fix_{Z}(g) \,\}\,|
&= |\, \Fix_{\Omega_{n}(z)}(g)\,|/ [\, G_{n}:H_{n}\,] \\
&= |\, \Fix_{\Omega_{n}(z)}(g)\,|/|\, \Omega_{n}(x)\,|.  
\end{align*}
The result now follows easily.
\end{proof}

Clearly the normalized permutation character $|\, \Fix_{\Omega_{n}(z)}(g)\,|/|\, \Omega_{n}(z)\,|$ 
is the probability that an element of $(\, \Omega_{n}(z), \mu_{n} \,)$ is fixed by $g \in G_{n}$, 
where $\mu_{n}$ is the uniform probability measure on $\Omega_{n}(z)$; and, in this sense, we can regard 
$G \curvearrowright (\, Z, \mu\,)$ as the ``limit'' of the sequence of finite permutation groups 
$G_{n} \curvearrowright (\, \Omega_{n}(z), \mu_{n} \,)$. Of course, the permutation group
$G_{n} \curvearrowright \Omega_{n}(z)$ is isomorphic to $G_{n} \curvearrowright G_{n}/H_{n}$,
where $G_{n}/H_{n}$ is the set of cosets of $H_{n} = \{\, h \in G_{n} \mid h \cdot z = z \,\}$ in $G_{n}$.
The following simple observation will play a key role in our later applications of Theorem \ref{T:point}. 

\begin{Prop} \label{P:conj}
If $H \leqslant S$ are finite groups and $\theta$ is the normalized permutation character
corresponding to the action $G \curvearrowright S/H$, then
\[
\theta(g) = \frac{|\, g^{S} \cap H\,|}{|\, g^{S}\,|} = 
\frac{|\, \{ s \in S \mid s g s^{-1} \in H \,\}|}{|S|}.
\]
\end{Prop}

\begin{proof}
Fix some $g \in S$. In order to see that the first equality holds, note that for each $a \in S$, 
we have that
\[
aH \in \Fix(g) \quad \Longleftrightarrow \quad a^{-1}g\,a \in H.
\]
Hence, counting the number of such $a \in S$, we see that
\[
|\,\Fix(g)\,|\,|\,H\,| = |\, g^{S} \cap H \,|\,|\, C_{S}(g)\,| 
                       = |\, g^{S} \cap H \,|\,|\, S \,|/|\, g^{S}\,|.
\]
It follows that 
\[
\theta(g) = |\, \Fix(g)\,|/[\,S:H\,] = |\, g^{S} \cap H\,|/|\, g^{S}\,|.
\]
To see that the second inequality holds, note that
\[
\frac{|\, \{ s \in S \mid s g s^{-1} \in H \,\}|}{|S|} =
\frac{|\, g^{S} \cap H\,|\, |\, C_{S}(g) \,| }{|\, g^{S}\,|\, |\, C_{S}(g) \,| }. 
\]
\end{proof}

\section{An outline of the proof of Theorem \ref{T:main}} \label{S:outline}

In this section, we will briefly outline the strategy of the proof of Theorem \ref{T:main}.
Let $(k_{n})$ be a sequence of natural numbers $k_{n} \geq 2$, let 
$X_{n} = \prod_{0 \leq m \leq n}[k_{m}]$, and let $G = \bigcup_{n \in \mathbb{N}}G_{n}$
be the corresponding SD$\mathcal{S}$-group with $G_{n} = \Sym(X_{n})$.
Let $\nu$ be an ergodic IRS of $G$. Then we can suppose that $\nu$ is not a 
Dirac measure $\delta_{N}$ concentrating on a normal subgroup $N \trianglelefteq G$. Let $\nu$ be
the stabilizer distribution of the ergodic action $G \curvearrowright (\, Z, \mu\,)$
and let $\chi(g) = \mu(\, \Fix_{Z}(g) \,)$ be the corresponding character. For each $z \in Z$ and 
$n \in \mathbb{N}$, let $\Omega_{n}(z) = \{\, g \cdot z \mid g \in G_{n}\,\}$.
Then, by Theorem \ref{T:point}, for $\mu$-a.e. $z \in Z$, for all $g \in G$,
\[
\mu(\, \Fix_{Z}(g) \,) = \lim_{n \to \infty} |\, \Fix_{\Omega_{n}(z)}(g)\,|/|\, \Omega_{n}(z)\,|.  
\]
Fix such an element $z \in Z$ and let $H = \{\, h \in G \mid h \cdot z = z \,\}$ be the corresponding point 
stabilizer. Clearly we can suppose that $z$ has been chosen so that if $g \in H$, then $\chi(g) > 0$.

For each $n \in \mathbb{N}$, let $H_{n} = H \cap G_{n}$. Then, examining the list of ergodic IRS's in
the statement of Theorem \ref{T:main}, we see that it is necessary to show that there exists a {\em fixed\/}
integer $r \geq 1$ such that for all but finitely many $n \in \mathbb{N}$, there is a subset
$U_{n} \subseteq X_{n}$ of cardinality $r$ such that $H_{n}$ fixes $U_{n}$ pointwise and induces at
least the alternating group on $X_{n} \smallsetminus U_{n}$. Most of our effort will devoted to
eliminating the possibility that $H_{n}$ acts transitively on $X_{n}$ for infinitely many 
$n \in \mathbb{N}$. In more detail, we will show that if $H_{n}$ acts transitively on $X_{n}$ for 
infinitely many $n \in \mathbb{N}$, then there exists an element $g \in H$ such that
\[
\mu(\, \Fix_{Z}(g) \,) = \lim_{n \to \infty} |\, g^{G_{n}} \cap H_{n}\,|/|\, g^{G_{n}}\,| =
|\, \{ s \in G_{n} \mid s g s^{-1} \in H_{n} \,\}|/|G_{n}| = 0,
\] 
which is a contradiction. Our analysis will split into three cases, depending on whether for
infinitely many $n \in \mathbb{N}$,
\begin{itemize}
\item[(i)] $H_{n}$ acts primitively on $X_{n}$; or
\item[(ii)] $H_{n}$ acts imprimitively on $X_{n}$ with a {\em fixed\/} maximal block-size $d$; or
\item[(iii)] $H_{n}$ acts imprimitively on $X_{n}$ with maximal blocksize $d_{n} \to \infty$.
\end{itemize}
Cases (i) and (ii) are easily dealt with via a straightforward counting argument based on
Stirling's Approximation, which will be presented in Section \ref{S:stirling}. Case (iii)
requires a more involved probabilistic argument which will be presented in
Section \ref{S:prob}. Essentially the same probabilistic argument will then show that there
exists a fixed integer $r \geq 1$ such that for all but finitely many $n \in \mathbb{N}$, there is an 
$H_{n}$-invariant subset $U_{n} \subseteq X_{n}$ of cardinality $r$ such that $H_{n}$ acts transitively on 
$X_{n} \smallsetminus U_{n}$. Repeating the above analysis for the action 
$H_{n} \curvearrowright X_{n} \smallsetminus U_{n}$, we will obtain that $H_{n}$ induces at
least the alternating group on $X_{n} \smallsetminus U_{n}$; and an easy application of
Nadkarni's Theorem on compressible group actions, which we will present in Section \ref{S:proof},
will show that $H_{n}$ fixes $U_{n}$ pointwise.

At this point, we will have shown that the ergodic IRS $\nu$ concentrates on the same space
of subgroups $\mathcal{S}_{r} \subseteq \Sub_{G}$ as one of the target IRS's $\sigma_{r}$
or $\tilde{\sigma}_{r}$. Finally, via another application of the Pointwise Ergodic Theorem,
we will show that the action of $G$ on $\mathcal{S}_{r}$ is uniquely ergodic and hence that
$\nu = \sigma_{r}$ or $\nu = \tilde{\sigma}_{r}$, as required.

\section{Almost primitive actions} \label{S:stirling}

In Sections \ref{S:stirling} and \ref{S:prob}, we will fix an element $1 \neq g \in \Sym(a)$ having a 
cycle decomposition consisting of $k_{i}$ $m_{i}$-cycles for $1 \leq i \leq t$, where each $k_{i} \geq 1$.
Let $\ell \gg a$ and let $\varphi: \Sym(a) \to S = \Sym(a \ell)$ be an $\ell$-fold diagonal 
embedding. Then identifying $g$ with $\varphi(g)$, the cyclic decomposition of $g \in S$ consists of 
$k_{i}\ell$ $m_{i}$-cycles for $1 \leq i \leq t$ and hence
\begin{equation} \label{E:conj}
|g^{S}| = \frac{(a \ell)!}{ \prod_{1 \leq i \leq t} (k_{i}\ell)! m_{i}^{k_{i}\ell}}.
\end{equation}
In this section, we will consider the value of the normalized permutation character
$|\, g^{S} \cap H\,|/|\, g^{S}\,|$ for the action $S \curvearrowright S/H$ as $\ell \to \infty$ 
in the cases when:
\begin{enumerate}
\item[(i)] $H$ is a primitive subgroup of $S$; or 
\item[(ii)] $H$ is an imprimitive subgroup of $S$ preserving a maximal system of
imprimitivity of fixed block-size $d$.
\end{enumerate}

In both cases, we will make use of the following theorem of Praeger-Saxl \cite{ps}. (It
is perhaps worth mentioning that the proof of Theorem \ref{T:ps} does not rely upon the
classification of the finite simple groups.)

\begin{Thm} \label{T:ps}
If $H < \Sym(n)$ is a primitive subgroup which does not contain $\Alt(n)$,
then $|H| < 4^{n}$.
\end{Thm}

We will also make use of the following variant of Stirling's Approximation:
\begin{equation} \label{E:stirling}
1 \leq \frac{n!}{\sqrt{2 \pi n} \left(\frac{n}{e}\right)^{n}}
\leq \frac{e}{\sqrt{2\pi}}  \quad \quad \text{ for all } n \geq 1.
\end{equation}

\begin{Lem} \label{L:conj}
There exist constants $r$, $s > 0$ (which only depend on the parameters 
$a, k_{1}, \cdots, k_{t}, m_{1}, \cdots, m_{t}$) such that
\[
|g^{S}| > r\, s^{\ell} \, \ell^{( a - \sum k_{i}) \ell} \geq r\, s^{\ell} \, \ell^{\,\ell}.
\] 
\end{Lem}

\begin{proof}
Combining (\ref{E:conj}) and (\ref{E:stirling}), it follows that there exists a constant $c > 0$
such that
\[
|g^{S}| > c \, 
\frac{\left( a \ell/e \right)^{a \ell} \sqrt{ 2 \pi a \ell}}
{\prod_{1 \leq i \leq t} m_{i}^{k_{i}\ell}\,
\prod_{1 \leq i \leq t} \left( k_{i} \ell/e \right)^{k_{i} \ell} \sqrt{ 2 \pi k_{i} \ell}}
\]
and this implies that 
\[
|g^{S}| > c\, d^{\ell} \, \ell^{( a - \sum k_{i}) \ell} \,
\frac{\sqrt{ 2 \pi a \ell}}{\prod_{1 \leq i \leq t} \sqrt{ 2 \pi k_{i} \ell}}
\]
for a suitably chosen constant $d > 0$. The result now follows easily.
\end{proof}

The following lemma will eliminate the possibility of primitive actions in the proof
of Theorem \ref{T:main}.

\begin{Lem} \label{L:prim}
For each $\varepsilon > 0$, there exists an integer $\ell_{\varepsilon}$ such that if
$\ell \geq \ell_{\varepsilon}$ and $H < S = \Sym(a \ell)$ is a primitive subgroup which does not contain 
$\Alt(a\ell)$, then $|\, g^{S} \cap H\,|/|\, g^{S}\,| < \varepsilon$.
\end{Lem}

\begin{proof}
Suppose that $\ell \gg a$ and that $H < S = \Sym(a\ell)$ is a primitive subgroup which does not 
contain $\Alt(a\ell)$. Applying Lemma \ref{L:conj} and Theorem \ref{T:ps}, there exist constants 
$r$, $s > 0$ such that 
\[
|\, g^{S} \cap H\,|/|\, g^{S}\,| < |\, H\,|/|\, g^{S}\,| < 
\frac{4^{a\ell}}{r\,s^{\ell} \ell^{\,\ell}}
\]
and the result follows easily.
\end{proof}

Next we consider the case when $H$ is an imprimitive subgroup of $S$ preserving a maximal system of
imprimitivity of fixed block-size $d$. Of course, we can suppose that 
$H \leqslant \Sym(d) \Wr \Sym( a\ell/d) < S = \Sym(a\ell)$. The following result
is another easy consequence of Stirling's Approximation.

\begin{Lem} \label{L:wreath}
There exist constants $b$, $c > 0$ (which only depend on the parameters $a$, $d$) such that
$|\Sym(d) \Wr \Sym( a\ell/d)| < b\, c^{\,\ell}\, \ell^{a \ell/d}$.
\end{Lem}

In this case, we can only show that $|\, g^{S} \cap H\,|/|\, g^{S}\,|$ is small for those elements 
$g \in \Sym(a)$ such that $a/d < a - \sum k_{i}$. Fortunately, clause (ii) of the following lemma will guarantee 
the existence of a ``suitable such'' element during the proof of Theorem \ref{T:main}.

\begin{Lem} \label{L:imp}
For each $d \geq 2$ and $\varepsilon > 0$, there exists an integer $\ell_{d,\varepsilon}$ such that if
$\ell \geq \ell_{d,\varepsilon}$ and $H < S = \Sym(a\ell)$ is an imprimitive subgroup with a {\em maximal\/}
system of imprimitivity $\mathcal{B}$ of blocksize $d$, then either:
\begin{enumerate}
\item[(i)] $|\, g^{S} \cap H\,|/|\, g^{S}\,| < \varepsilon$; or
\item[(ii)] the induced action of $H$ on $\mathcal{B}$ contains $\Alt(\mathcal{B})$.
\end{enumerate}
\end{Lem}

\begin{proof}
Suppose that $\ell \gg a$ and that $H < S = \Sym(a\ell)$ is an imprimitive subgroup with a maximal
system of imprimitivity $\mathcal{B}$ of blocksize $d$. Let $\Gamma \leqslant \Sym(\mathcal{B})$ be the
group induced by the action of $H$ on $\mathcal{B}$ and suppose that $\Gamma$ does not contain
$\Alt(\mathcal{B})$. Since $\mathcal{B}$ is a maximal system of imprimitivity, it follows that
$\Gamma$ is a primitive subgroup of $\Sym(\mathcal{B})$; and hence by Theorem \ref{T:ps}, we obtain
that $|\,\Gamma\,| < 4^{a\ell/d}$. Since $H$ is isomorphic to a subgroup of 
$\Sym(d) \Wr \Gamma$, it follows that
\[
|\,H\,| < (\,d!\,)^{a\ell/d} 4^{a\ell/d} = c^{\,\ell},
\]   
where $c = (\, d!\,4\,)^{a/d}$. Arguing as in the proof of Lemma \ref{L:prim}, the result
follows easily.
\end{proof}

\section{Imprimitive and intransitive actions} \label{S:prob}

In this section, we will continue to fix an element $1 \neq g \in \Sym(a)$. Suppose that 
the cycle decomposition of $g \in \Sym(a)$ has $k$ nontrivial cycles. Let $\ell \gg a$ and let 
$\varphi: \Sym(a) \to S = \Sym(a \ell)$ be an $\ell$-fold diagonal embedding. Then identifying $g$ with $\varphi(g)$, 
the cyclic decomposition of $g \in S$ has $k \ell$ nontrivial cycles. In this section, we will consider the 
value of the normalized permutation character for the action $S \curvearrowright S/H$ as $\ell \to \infty$ in the 
following two cases.
\begin{enumerate}
\item[(i)] There exists an $H$-invariant subset $U \subset [a\ell]$ of {\em fixed\/} cardinality $r \geq 0$ such
that $H$ acts imprimitively on $T = [a\ell] \smallsetminus U$ with a proper system of imprimitivity $\mathcal{B}$ of 
blocksize $d$ with $d \to \infty$ as $\ell \to \infty$.
\item[(ii)] $H$ is an intransitive subgroup of $S$ with an $H$-invariant subset $U \subset [a\ell]$
of cardinality $r = |U| \leq a\ell/2$ such that $r \to \infty$ as $\ell \to \infty$.
\end{enumerate}

Our approach is this section will be probabilistic; i.e. we will regard the normalized permutation character 
$|\, \{ s \in S \mid s g s^{-1} \in H \,\}|/|S|$ as the probability that a uniformly random permutation
$s \in S$ satisfies $s g s^{-1} \in H$. Our probability theoretic notation is standard. In particular, 
if $E$ is an event, then $\mathbb{P}\,[E]$ denotes the corresponding probability and $1_{E}$
denotes the indicator function; and if $N$ is a random variable, 
then $\mathbb{E}\,[N]$ denotes the expectation, $\Var [N]$ denotes the variance and $\sigma = (\Var [N])^{1/2}$
denotes the standard deviation. We will make use of the following easy consequence of Chebyshev's inequality.

\begin{Lem} \label{L:bigO} 
Let $(N_{\ell})$ be a sequence of non-negative random variables such that
$\mathbb{E}\, [N_{\ell} ] = \mu _{\ell} >0$ and $\Var [N_{\ell}]=\sigma _{\ell}^2 >0$.
If $\lim _{\ell \to \infty} \mu _{\ell}/\sigma _{\ell}  =\infty$, then 
$\mathbb{P}\, [ N_{\ell} > 0 ] \to 1$ as $\ell \to \infty$.
\end{Lem}

\begin{proof}
Let $K_{m} =(\mu _{\ell} / \sigma _{\ell} ) ^{1/2}$ and let $L_{\ell} = \mu_{\ell} - K_{\ell} \sigma _{\ell}$. 
By Chebyshev's inequality, 
\[
\mathbb{P}\, [N_{\ell} > L_{\ell} ] \geq 1-\frac{1}{K_{\ell}^2} 
\]
and so $\mathbb{P}\, [N_{\ell} > L_{\ell} ] \to 1$ as $\ell \to \infty$.
In addition, for all sufficiently large $\ell$, we have that $K_{\ell} > 1$ and hence 
$L_{\ell} = (\mu _{\ell} \sigma _{\ell} )^{1/2}(K_{\ell} - 1 ) >0$.
\end{proof}

In our arguments, it will be convenient to make use of big O notation. Recall that if $(a_{m} )$ and $(x_{m})$ 
are sequences of real numbers, then $a_{m} = O(x_{m})$ means that there exists a constant $C>0$ and 
an integer $m_{0}\in \mathbb{N}$ such that $|a_{m}| \leq C|x_{m}|$ for all $m\geq m_{0}$. Also if
$(c_{m})$ is another sequence of real numbers, then we write $a_{m} = c_{m} + O(x_{m})$ to mean that 
$a_{m}-c_{m} = O(x_{m})$.

\begin{Lem} \label{L:block}
For each $r \geq 0$ and $\varepsilon > 0$, there exists an integer $d_{r,\varepsilon}$ such that if
$d \geq d_{r,\varepsilon}$ and $H < S = \Sym(a\ell)$ is a subgroup such that:
\begin{enumerate}
\item[(i)] there exists an $H$-invariant subset $U \subset [a\ell]$ of cardinality $r$, and
\item[(ii)] $H$ acts imprimitively on $T = [a\ell] \smallsetminus U$ with a proper
system of imprimitivity $\mathcal{B}$ of blocksize $d$,
\end{enumerate}
then 
$|\{\, s \in S \mid s gs ^{-1} \in H \,\} |/|S| < \varepsilon$.
\end{Lem}

\begin{proof}
Let $m = a\ell$ and let $Z \subset [m]$ be a subset which contains one element from every non-trivial cycle of $g$. 
Then $|Z| = cm$ where $0< c = k/a \leq 1/2$ is a fraction which is independent of $\ell$. Let $Y=g(Z)$ so that 
$Z\cap Y=\emptyset$. Fix an element $z_{0} \in Z$ and let $y_{0} = g(z_{0})$. Let $s \in S$ be a uniformly 
random permutation. If $s(z_{0})$,$s(y_{0}) \in T$, let $B_0$, $C_0 \in \mathcal{B}$ be the blocks 
in $\mathcal{B}$ containing $s(z_{0})$ and $s(y_{0})$ respectively; otherwise, let 
$B_{0} = C_{0} = \emptyset$. Let $E$ be the event that either $s(z_{0}) \notin T$ or $s(y_{0}) \notin T$.
Then clearly $\mathbb{P}\, [\,E\,] \leq 2r/m$. Let 
\[
J(s) = \{ z \in Z \smallsetminus \{ z_{0}\} \mid s(z)\in B_{0} \text{ and } s(g(z))\notin C_{0} \}.
\]
Note that if $J(s) \neq \emptyset$, then $s g s^{-1}(B_{0})$ intersects at least two of the blocks of
$\mathcal{B}$ and thus $s g s^{-1} \notin H$. Hence it suffices to show that
\begin{equation} \label{E:J>0}
\mathbb{P}\, [\, |J(s)| >0 \,] \to 1 \quad \text{ as } d\to \infty.
\end{equation}

Since we will be using Lemma \ref{L:bigO}, we need to compute the asymptotics of the expectation and variance of
the random variable 
\[
|J(s)| = \sum _{z \in Z\setminus \{ z_0 \}} 1_{[ s(z) \in B_{0}, \ s(g(z)) \notin C_{0} ]}.
\]
We will often implicitly use that $0<\tfrac{d}{n} \leq \tfrac{1}{2}$, and that 
$r=O(1)$, $\tfrac{d}{n} = O(1)$ and $\tfrac{n-d}{n} = O(1)$.
To compute both the expectation $\mathbb{E}\, [\,|J(s) |\,]$ and the second moment
$\mathbb{E}\, [\,|J(s)|^{2}\,]$ , we will separately condition on the events $E$,
$[\,C_{0} = B_{0}\,] \smallsetminus E$ and $[\,C_{0} \neq B_{0}\,] \smallsetminus E$. 
On $E$ the expectation is $0$, and otherwise we have that
\begin{align*}
\mathbb{E}\, \big[\, |J(s)|\, \big|\, [C_{0}=B_{0}] \smallsetminus E \,\big] 
&=\textstyle\sum _{z\in Z\smallsetminus \{ z_{0} \} }\mathbb{P}\, \big[\, s(z)\in B_{0}, 
\ s(g(z))\notin B_{0} \,\big|\, s(y_{0})\in B_{0}\, \big] \\
&= (cm-1)\frac{(d-2)(m-d)}{(m-2)(m-3)} \\
&= cd(1-\tfrac{d}{m}) + O(1) 
\end{align*}
and that
\begin{align*}
\mathbb{E}\, \big[\, |J(s)| \,\big|\, [C_{0}\neq B_{0}] \smallsetminus E \,\big] 
&= \textstyle \sum _{z\in Z\smallsetminus \{ z_0\} } 
\mathbb{P}\, \big[\, s(z)\in B_{0}, \ s(g(z))\notin C_{0} \,\big|\, s(y_{0})\notin B_{0}\, \big] \\
&= (cm-1)\frac{(d-1)(m-(d-2))}{(m-2)(m-3)} \\
&= cd(1-\tfrac{d}{m}) + O(1).
\end{align*}
Since $\mathbb{P}\, [\,E\,]\leq 2r/m$, we have that 
$cd(1-\tfrac{d}{m})(1- \mathbb{P}\, [\,E\,] ) = cd(1-\tfrac{d}{m}) + O(1)$ and hence we obtain that
\begin{equation}\label{E:expect}
\mathbb{E}\, [\,|J(s)|\,] = cd(1-\tfrac{d}{m}) + O(1 ) 
\end{equation}
and that
\begin{equation} \label{E:expect-square}
\mathbb{E}\, [\,|J(s)|\,]^2 = [cd(1-\tfrac{d}{m})]^2 + O(d),
\end{equation}
where in the second equality we use the fact that $\mathbb{E}\, [\,|J(s)|\, ] = O(d)$. 
For the second moment, we have that
\begin{align*}
\mathbb{E}\, &\big[\, |J(s)|^2\, \big|\, [C_{0}=B_{0}] \smallsetminus E \,\big] -
\mathbb{E}\, \big[\, |J(s)| \,\big|\, [C_{0}=B_{0}] \smallsetminus E \,\big] \\
&= \sum _{w\neq z\in Z\smallsetminus \{ z_{0}\} }\mathbb{P}\,\big[\, s(w),  s(z)\in B_{0},
 \ s(g(w)), s(g(z))\notin B_{0}  \,\big|\, s(y_{0})\in B_{0} \,\big] \\
&= (cm-1)(cm-2)\frac{(d-2)(d-3)(m-d)(m-(d-1))}{(m-2)(m-3)(m-4)(m-5)} \\
&= [cd(1-\tfrac{d}{m} ) ] ^2 + O(d) 
\end{align*}
and
\begin{align*}
\mathbb{E}\, &\big[\, |J(s)|^2 \,\big|\, [C_{0}\neq B_{0}] \smallsetminus E \,\big] -
\mathbb{E}\, \big[\, |J(s)|\, \big|\, [C_{0}\neq B_{0}] \smallsetminus E \,\big]\\ 
&= \sum _{w\neq z\in Z\smallsetminus \{ z_{0}\} }\mathbb{P}\,\big[\, s(w), s(z)\in B_{0}, 
\ s(g(w)), s(g(z))\notin C_{0}\,  \big|\, s(y_{0})\notin B_{0} \,\big] \\
&= (cm-1)(cm-2)\frac{(d-1)(d-2)(m-(d-1))(m-(d-2))}{(m-2)(m-3)(m-4)(m-5)} \\
&=[cd(1-\tfrac{d}{m} ) ] ^2 + O(d). 
\end{align*}
Since $[cd(1-\tfrac{d}{n})]^2(1-\mathbb{P}\, [E] ) = [cd(1-\tfrac{d}{n})]^2 + O(d)$, it follows that
\begin{equation}\label{E:Var}
\mathbb{E}\, [\,|J(s)|^2 \,] = \big[ cd (1-\tfrac{d}{m}) \big]^2  + O(d).
\end{equation}
Combining (\ref{E:expect-square}) and (\ref{E:Var}) we obtain that
\[ 
\Var(|J(s)|) = \mathbb{E}\, [\,|J(s)|^2\,]-\mathbb{E}\, [\,|J(s)|\,]^2 = O(d),
\] 
and hence $\Var(|J(s)|)^{1/2} = O(\sqrt{d} )$. Of course, (\ref{E:expect}) implies that 
$d = O( \mathbb{E}\,[\,|J(s)|\,])$. Thus there exists a constant $C > 0$ such that
$\sigma = \Var(|J(s)|)^{1/2} \leq C \sqrt{d}$ and $d \leq C\, \mathbb{E}\,[\,|J(s)|\,]) = C \mu$
for all sufficiently large $d$. It follows that 
\[
\mu/\sigma \geq C^{-1}d/C \sqrt{d} = C^{-2} \sqrt{d} \to \infty \quad \quad 
\text{ as } d \to \infty.
\]
Applying Lemma \ref{L:bigO}, we conclude that $\mathbb{P}\, [\,|J(s)|>0 \,] \to 1$ as $d \to \infty$, 
which proves (\ref{E:J>0}). This completes the proof of Lemma \ref{L:block}.
\end{proof}

\begin{Lem} \label{L:int}
For each $\varepsilon > 0$, there exists an integer $r_{\varepsilon}$ such that if
$r \geq r_{\varepsilon}$ and $H < S = \Sym(a\ell)$ is an intransitive subgroup with
an $H$-invariant subset $U\subset [a\ell]$ such that $r = |U|\leq a\ell/2$, then
$|\{\, s \in S \mid s g s^{-1} \in H \,\} |/|S| < \varepsilon$.
\end{Lem}

\begin{proof}
The proof is similar to that of Lemma \ref{L:block}, although the computations are simpler. 
Let $m = a \ell$ and let $Z$, $Y$, and $c$ be as in Lemma \ref{L:block}. Fix some 
$H$-invariant $U\subseteq [m]$ with $|U|=r$. Let $s \in S$ be a uniformly random permutation and let
\[ 
I(s) = \{\, z\in Z \mid s(z) \in U \text{ and } s(g(z)) \notin U \,\}.
\] 
If $I(s) \neq \emptyset$, then $U$ is not $s g s^{-1}$-invariant and thus $s g\ s^{-1} \notin H$. 
Hence it suffices to show that
\begin{equation}\label{eqn:I>0}
\mathbb{P}\, [\,|I(s)|>0 \,] \to 1 \text{ as } r \to \infty.
\end{equation}
Computations similiar to those in the proof of Lemma \ref{L:block} show that
\begin{equation} \label{E:similar} 
\mathbb{E}\, [\,|I(s)|\,] = cr (1-\tfrac{r}{m}) + O(1) \text { and }
\mathbb{E}\, [\,|I(s)|^2\,] = [cr(1-\tfrac{r}{m})]^2 + O(r).
\end{equation}
It follows that $\Var(|I(s)|)^{1/2} = O(\sqrt{r} )$ and $r = O( \mathbb{E}\,[\,|I(s)|\,])$;
and another application of Lemma \ref{L:bigO} shows that $\mathbb{P}\, [\,|I(s)|>0 \,] \to 1$ as $r \to \infty$.
\end{proof}

\section{The proof of Theorem \ref{T:main}} \label{S:proof}

In this section, we will present the proof of Theorem \ref{T:main}. 
Let $(k_{n})$ be a sequence of natural numbers $k_{n} \geq 2$, let 
$X_{n} = \prod_{0 \leq m \leq n}[k_{m}]$, and let $G = \bigcup_{n \in \mathbb{N}}G_{n}$
be the corresponding SD$\mathcal{S}$-group with $G_{n} = \Sym(X_{n})$.
Let $\nu$ be an ergodic IRS of $G$. Then we can suppose that $\nu$ is not a 
Dirac measure $\delta_{N}$ concentrating on a normal subgroup $N \trianglelefteq G$. 
Applying Creutz-Peterson \cite[Proposition 3.3.1]{cp}, let $\nu$ be
the stabilizer distribution of the ergodic action $G \curvearrowright (\, Z, \mu\,)$
and let $\chi(g) = \mu(\, \Fix_{Z}(g) \,)$ be the corresponding character. Then, since $\nu \neq \delta_{1}$,
it follows that $\chi \neq \chi_{\reg}$, where $\chi_{\reg}$ is the {\em regular character\/} defined by
\[
\chi_{\reg}(g) = 
\begin{cases}
1 &\text{if } g = 1; \\
0 &\text{if } g \neq 1.
\end{cases}
\]
For each $z \in Z$ and 
$n \in \mathbb{N}$, let $\Omega_{n}(z) = \{\, g \cdot z \mid g \in G_{n}\,\}$.
Then, by Theorem \ref{T:point}, for $\mu$-a.e. $z \in Z$, for all $g \in G$, we have that
\begin{equation} \label{E:generic}
\mu(\, \Fix_{Z}(g) \,) = \lim_{n \to \infty} |\, \Fix_{\Omega_{n}(z)}(g)\,|/|\, \Omega_{n}(z)\,|.  
\end{equation}
Fix such an element $z \in Z$ and let $H = \{\, h \in G \mid h \cdot z = z \,\}$ be the corresponding point 
stabilizer. Clearly we can suppose that the element $z \in Z$ has been chosen so that $H$ is not 
a normal subgroup of $G$ and so that if $g \in H$, then $\chi(g) > 0$.
For each $n \in \mathbb{N}$, let $H_{n} = H \cap G_{n}$. Then, by Proposition \ref{P:conj}, for each $g \in G$, 
we have that
\[
\mu(\, \Fix_{Z}(g) \,) = \lim_{n \to \infty} |\, g^{G_{n}} \cap H_{n}\,|/|\, g^{G_{n}}\,| =
|\, \{ s \in G_{n} \mid s g s^{-1} \in H_{n} \,\}|/|G_{n}|.
\] 
We will consider the various possibilities for the action of $H_{n} \leqslant G_{n} = \Sym(X_{n})$
on the set $X_{n}$. 

\begin{Lem} \label{L:intransitive}
There exist only finitely many $n \in \mathbb{N}$ such that $H_{n}$ acts transitively on $X_{n}$.
\end{Lem}

\begin{proof}
Suppose, on the contrary, that $T = \{\, n \in \mathbb{N} \mid H_{n} \text{ acts transitively on } X_{n} \,\}$ 
is infinite. First consider the case when there are infinitely many $n \in T$ such that 
$H_{n}$ acts primitively on $X_{n}$. If $\Alt(X_{n}) \leqslant H_{n}$ for infinitely many $n \in T$,
then either $H = G$ or $H = A(G)$, which contradicts the fact that $H$ is not a normal
subgroup of $G$; and hence there are 
only finitely many such $n \in T$. But then Lemma \ref{L:prim} implies that 
\[
\chi(g) = \lim_{n \to \infty} |\, g^{G_{n}} \cap H_{n}\,|/|\, g^{G_{n}}\,| = 0
\]
for all $1 \neq g \in G$, which contradicts the fact that $\chi \neq \chi_{\reg}$. Thus $H_{n}$ acts imprimitively
on $X_{n}$ for all but finitely many $n \in T$. For each $d \geq 2$, let $B_{d}$ be the set
of $n \in T$ such that:
\begin{enumerate}
\item[(a)] each nontrivial proper system of imprimitivity for the action of $H_{n}$ on $X_{n}$
has blocksize at most $d$; and
\item[(b)] there exists a (necessarily maximal) system of imprimitivity $\mathcal{B}_{n}$ of blocksize $d$ 
for the action of $H_{n}$ on $X_{n}$. 
\end{enumerate}
If each $B_{d}$ is finite, then Lemma \ref{L:block} (in the case when $r = 0$) implies that
\[
\chi(g) = \lim_{n \to \infty} |\, \{ s \in G_{n} \mid s g s^{-1} \in H_{n} \,\}|/|G_{n}|= 0
\]
for all $1 \neq g \in G$, which again contradicts the fact that $\chi \neq \chi_{\reg}$. 
Thus there exists $d \geq 2$ such that $B_{d}$ is infinite. Since $\chi \neq \chi_{\reg}$,
Lemma \ref{L:imp} implies that for all but finitely many $n \in B_{d}$, the induced action
of $H_{n}$ on $\mathcal{B}_{n}$ contains $\Alt(\mathcal{B}_{n})$. Choose some $m \in B_{d}$ such
that $|X_{m}| = a \gg d$ and let $g \in H_{m}$ induce either an $a/d$-cycle or an
$(a/d - 1)$-cycle on $\mathcal{B}_{m}$ (depending upon whether $a/d$ is odd or even). 
Suppose that $g$ has a cycle decomposition as an element of $\Sym(X_{n})$ consisting of
$k_{i}$ $m_{i}$-cycles for $1 \leq i \leq t$, where each $k_{i} \geq 1$. Since $g$ has at most $2d$
orbits on $X_{m}$, it follows that $\sum k_{i} \leq 2d$ and so $a - \sum k_{i} > a/d$.
Next suppose that $n \in B_{d}$ and that $n \gg m$. Let $|X_{n}| = a\ell$. Then, applying
Lemma \ref{L:wreath}, there exist constants $b$, $c > 0$ such that
\[
|H_{n}| \leq |\Sym(d) \Wr \Sym( \ell a/d)| < b\, c^{a \ell}\, \ell^{a \ell/d}.
\]
Also, by Lemma \ref{L:conj}, there exist constants $r$, $s > 0$ such that 
$|g^{G_{n}}| > r\, s^{a \ell} \, \ell^{( a - \sum k_{i}) \ell}$. Since
$a - \sum k_{i} > a/d$, it follows that
\[
\chi(g) = \lim_{n \to \infty} |\, g^{G_{n}} \cap H_{n}\,|/|\, g^{G_{n}}\,| \leq
\lim_{n \to \infty} |\, H_{n}\,|/|\, g^{G_{n}}\,| = 0,
\]
which contradicts the fact that $\chi(g) > 0$ for all $g \in H$.
\end{proof}

Let $I = \{\, n \in \mathbb{N} \mid H_{n} \text{ acts intransitively on } X_{n} \,\}$; and for each
$n \in I$, let
\[
r_{n}= \max \{\, |U| :  U\subseteq X_{n} \text{ is }H_{n}\text{-invariant and }|U|\leq \tfrac{1}{2}|X_{n}| \,\}.
\]
Then each $r_{n} \geq 1$. Furthermore, since $\chi \neq \chi_{\reg}$, Lemma \ref{L:int} implies that the sequence 
$(\, r_{n}\,)_{n \in I}$ is bounded above. Let $r = \lim \inf r_{n}$ and let $I_{r}$ be the set
of integers $n \in I$ such that:
\begin{enumerate}
\item[(i)] $r_{n} = r$;
\item[(ii)] $n > \max \{\, m \in I \mid r_{m} < r \,\}$; and
\item[(iii)] $n > r+1$. 
\end{enumerate}
Here we have chosen $n > r+1$ in order to ensure that $|X_{n}| > 4r$.  

\begin{Lem} \label{L:trans}
If $n \in I_{r}$, then there exists a unique $H_{n}$-invariant subset $U_{n} \subset X_{n}$ of
cardinality $r$ and $H_{n}$ acts transitively on $X_{n} \smallsetminus U_{n}$.
\end{Lem} 

\begin{proof}
By definition, there exists at least one $H_{n}$-invariant subset $U_{n} \subset X_{n}$ of
cardinality $r$. Suppose that $V \subseteq X_{n} \smallsetminus U_{n}$ is an $H_{n}$-orbit. If
$|V| \leq \tfrac{1}{2}|X_{n}|$, then $|V| \leq r$ and so the $H_{n}$-invariant subset 
$U_{n} \cup V$ satisfies $r < |U_{n} \cup V| \leq 2r < \tfrac{1}{2}|X_{n}|$, which is a contradiction. 
Thus each $H_{n}$-orbit $V \subseteq X_{n} \smallsetminus U_{n}$ satisfies $|V| > \tfrac{1}{2}|X_{n}|$, 
and this clearly implies that $H_{n}$ acts transitively on $X_{n} \smallsetminus U_{n}$.
\end{proof}

For each $n \in I_{r}$, let
$K_{n} = \{\, g \in H_{n} \mid g \cdot u = u \text{ for all } u \in U_{n} \,\}$ be
the pointwise stabilizer of $U_{n}$ and
let $Y_{n} = X_{n} \smallsetminus U_{n}$. As usual,
we will identify $K_{n}$ with the corresponding subgroup of $\Sym(Y_{n})$. 

\begin{Lem} \label{L:alt}
If $n \in I_{r}$, then $\Alt(Y_{n}) \leqslant K_{n} \leqslant \Sym(Y_{n})$.
\end{Lem}

\begin{proof}
Let $\Bar{H}_{n}$ be the subgroup of $\Sym(Y_{n})$ induced by the action of $H_{n}$ on $Y_{n}$.
Then, arguing as in the proof of Lemma \ref{L:intransitive}, we see first that 
$\Bar{H}_{n}$ must act primitively on $Y_{n}$ and then that $\Alt(Y_{n}) \leqslant \Bar{H}_{n}$.
Let $\pi_{n}: H_{n} \to \Sym(U_{n})$ be the homomorphism defined by $g \mapsto g \res U_{n}$.
Then $K_{n} = \ker \pi_{n} \trianglelefteq H_{n}$; and identifying $K_{n}$ with the corresponding
subgroup of $\Sym(Y_{n})$, we have that $K_{n} \trianglelefteq \Bar{H}_{n}$. Since
$|Y_{n}| = |X_{n}| - r > 3r$, it follows that 
$[\, \Bar{H}_{n}: K_{n}\,] \leq r! < |\Alt(Y_{n})|$ and hence $\Alt(Y_{n}) \leqslant K_{n}$.
\end{proof}

From now on, let $n_{0} = \min I_{r}$.

\begin{Lem} \label{L:cofinite}
\begin{enumerate}
\item[(i)] $I_{r} = \{\, n \in \mathbb{N} \mid n \geq n_{0}\,\}$.
\item[(ii)] For each $n \in I_{r}$ and $w \in U_{n}$, there exists a unique
$i \in [k_{n+1}]$ such that $w\sphat\, i \in U_{n+1}$. 
\end{enumerate}
\end{Lem}

\begin{proof}
Suppose that $n_{0} < n \in I_{r}$. Let $U = \{\, u \res n-1 \mid u \in U_{n} \,\} \subseteq X_{n-1}$ and let
\[
U^{+}_{n} = \{\, w \in X_{n} \mid (\, \exists u \in U_{n}\,)\,\, w \res n-1 = u \res n-1 \,\}.
\]
Then $\Alt( X_{n} \smallsetminus U_{n}^{+}) \leqslant K_{n} \leqslant H_{n}$ and this implies that
$\Alt(X_{n-1} \smallsetminus U) \leqslant H_{n-1}$. Thus $r_{n-1} \leq |U| \leq |U_{n}| = r$.
Since $n-1 \geq n_{0} > \max \{\, m \in I \mid r_{m} < r \,\}$, it follows that $r_{n-1} = r$
and this implies that the map $u \mapsto u \res n-1$ from $U_{n}$ to $U$ is injective. The result follows.
\end{proof}

Let $X = \prod_{n \geq 0} [k_{n}]$ and let 
$\mathcal{B}_{H} = \{\, x \in X \mid x \res n \in U_{n} \text{ for all } n \in I_{r}\,\}$.
Then clearly $\mathcal{B}_{H}$ is an element of the standard Borel space $[\,X\,]^{r}$
of $r$-element subsets of $X$; and it is easily checked that $\mathcal{B}_{H}$ is precisely the
set of $x \in X$ such that the corresponding orbit $H \cdot x$ is finite. It follows that
the Borel map $H \overset{\varphi}{\mapsto} \mathcal{B}_{H}$, defined on the $\nu$-measure 1
subset of those $H \in \Sub_{G}$ such that
\begin{enumerate}
\item[(a)] $H$ satisfies (\ref{E:generic}),
\item[(b)] $H$ is not a normal subgroup of $G$, and 
\item[(c)] if $g \in H$, then $\chi(g) > 0$,
\end{enumerate}
is $G$-equivariant; and so $\varphi_{*}\nu$ is an ergodic $G$-invariant probability measure on $[\,X\,]^{r}$.

Suppose now that there exists $n \in I_{r}$ such that $H_{n}$ acts nontrivially on $U_{n}$; say,
$u \neq g \cdot u = v$, where $u$, $v \in U_{n}$ and $g \in H_{n}$. Then, regarding $g$ as an element
of $H_{n+1}$, we have that $u\sphat\,i \neq g \cdot u\sphat\,i = v\sphat\,i$ for all $i \in [k_{n+1}]$
and it follows that there exists $i \in [k_{n+1}]$ such that $u\sphat\,i, v\sphat\,i \in U_{n+1}$. 
Continuing in this fashion, we see that there exist $x \neq y \in \mathcal{B}_{H}$ such that
$x \mathbin{E_{0}} y$, where $E_{0}$ is the Borel equivalence relation defined on $X$ by
\[
x \mathbin{E_{0}} y \quad \Longleftrightarrow \quad 
x(n) = y(n) \text{ for all but finitely many } n \in \mathbb{N}.
\]
Consequently, the following result implies that $\nu$ concentrates on those $H \in \Sub_{G}$
such that $H_{n}$ acts trivially on $U_{n}$ for all $n \in I_{r}$.

\begin{Lem} \label{L:compress}
There does {\em not\/} exist a $G$-invariant Borel probability measure on the standard Borel space
$S = \{\, F \in [\,X\,]^{r} : E_{0} \res F \text{ is not the identity relation }\}$.
\end{Lem}

Before proving Lemma \ref{L:compress}, we will complete the proof of Theorem \ref{T:main}. Continuing
our analysis of the $\nu$-generic subgroup $H < G$, we can suppose that $H_{n}$ acts trivially on
$U_{n}$ for all $n \in I_{r}$. Hence, applying Lemmas \ref{L:alt} and \ref{L:cofinite}, we see that
$\Alt(Y_{n}) \leqslant H_{n} \leqslant \Sym(Y_{n})$ for all $n \geq n_{0}$. If $n_{0} \leq n < m$ and
$k_{m}$ is even, then identifying $G_{n} = \Sym(X_{n})$ with the corresponding subgroup of
$G_{m} = \Sym(X_{m})$, we have that $\Sym(Y_{n}) \leq \Alt(Y_{m})$ and so $H_{n} = \Sym(Y_{n})$. 
In particular, if $G$ is a simple $SD\mathcal{S}$-group, then $H_{n} = \Sym(Y_{n})$ for all $n \geq n_{0}$.
Similarly, if $G$ is not simple, then either $H_{n} = \Sym(Y_{n})$ for all $n \geq n_{0}$, or else
$H_{n} = \Alt(Y_{n})$ for all but finitely many $n \geq n_{0}$. 

\begin{Notation} \label{N:two}
For each finite set $Y$, let $S^{+}(Y) = \Sym(Y)$ and $S^{-}(Y) = \Alt(Y)$.
\end{Notation}

\begin{Def} \label{D:target}
For each $r \geq 1$ and $\varepsilon = \pm$, let $\mathcal{S}_{r}^{\varepsilon}$ be the 
standard Borel space of subgroups $H < G$ such that there is an integer $n_{0}$ such that for all 
$n \geq n_{0}$, there exists a subset $U_{n} \subset X_{n}$ of cardinality $r$ such that 
$H_{n} = S^{\varepsilon}(X_{n} \smallsetminus U_{n})$.
\end{Def}

Summing up, we have shown that there exists $r \geq 1$ and $\varepsilon = \pm$ such that
the ergodic IRS $\nu$ concentrates on $\mathcal{S}_{r}^{\varepsilon}$. Thus the following
lemma completes the proof of Theorem \ref{T:main}.

\begin{Prop} \label{P:unique}
$\sigma_{r}$, $\tilde{\sigma}_{r}$ are the unique ergodic probability measures
on $\mathcal{S}^{+}_{r}$, $\mathcal{S}^{-}_{r}$ under the conjugation action of $G$.
\end{Prop}

\begin{proof}
Suppose, for example, that $m$ is an ergodic probability measure on $\mathcal{S}^{+}_{r}$. 
Then it is enough to show that if $B \subseteq \Sub_{G}$ is a basic clopen subset, then
$m(B) = \sigma_{r}(B)$. Let
$B = \{\, K \in \Sub_{G} \mid K \cap G_{m} = L \,\}$, 
where $m \in \mathbb{N}$ and $L \leqslant G_{m}$ is a subgroup. By the Pointwise Ergodic Theorem, 
there exists $H \in \mathcal{S}^{+}_{r}$ such that
\begin{align*}
m(B) &= \lim_{n \to \infty} |\,\{\, g \in G_{n} \mid g H g^{-1} \in B\,\}\,|/|G_{n}| \\
     &= \lim_{n \to \infty} |\,\{\, g \in G_{n} \mid g H_{n} g^{-1} \cap G_{m} = L \,\}\,|/|G_{n}|. 
\end{align*}
Similarly, there exists $H^{\prime} \in \mathcal{S}^{+}_{r}$ such that
\[
\sigma_{r}(B) = \lim_{n \to \infty} |\,\{\, g \in G_{n} \mid g H^{\prime}_{n} g^{-1} \cap G_{m} = L \,\}\,|/|G_{n}|. 
\]
Since $H$, $H^{\prime} \in \mathcal{S}^{+}_{r}$, there exists $a \in \mathbb{N}$ such that
$H_{n}$ and $H^{\prime}_{n}$ are conjugate in $G_{n}$ for all $n \geq a$; and this implies that
\begin{multline*}
\lim_{n \to \infty} |\,\{\, g \in G_{n} \mid g H_{n} g^{-1} \cap G_{m} = L \,\}\,|/|G_{n}| \\ 
= \lim_{n \to \infty} |\,\{\, g \in G_{n} \mid g H^{\prime}_{n} g^{-1} \cap G_{m}= L \,\}\,|/|G_{n}|. 
\end{multline*}
\end{proof}

The remainder of this section will be devoted to the proof of Lemma \ref{L:compress}. First we need to
recall Nadkarni's Theorem on compressible group actions. (Here we follow 
Dougherty-Jackson-Kechris \cite{djk}.) Suppose that $\Gamma$ is a countable discrete group and that 
$E = E^{Y}_{\Gamma}$ is the orbit equivalence relation of a Borel action of $\Gamma$ on a standard Borel space $Y$.
Then $[[E]]$ denotes the set of Borel bijections $f: A \to B$, where $A$, $B \subseteq Y$ are Borel
subsets, such that $f(y) \mathbin{E} y$ for all $y \in A$. If $A$, $B \subseteq Y$ are Borel subsets, then 
we write $A \sim B$ if there exists an $f \in [[E]]$ with $f: A \to B$; and we write $A \preceq B$ if there exists 
a Borel subset $B^{\prime} \subseteq B$ with $A \sim B^{\prime}$.
The usual Schr\"{o}der-Bernstein argument shows that
\[
A \sim B \quad \Longleftrightarrow \quad A \preceq B \text{ and } B \preceq A.
\]
The orbit equivalence relation $E$ is {\em compressible\/} if there exists a Borel
subset $A \subseteq Y$ such that $Y \sim A$ and $Y \smallsetminus A$ intersects
every $E$-class. We will make use of the easy direction of the following theorem;
i.e. the observation that (ii) implies (i).

\begin{Thm}[Nadkarni \cite{n}] \label{T:comp}
If $E = E^{Y}_{\Gamma}$ is the orbit equivalence relation of a Borel action of a countable group $\Gamma$ 
on a standard Borel space $Y$, then the following are equivalent:
\begin{enumerate}
\item[(i)] $E = E^{Y}_{\Gamma}$ is not compressible.
\item[(ii)] There exists a $\Gamma$-invariant Borel probability measure on $Y$.
\end{enumerate}
\end{Thm}

Thus to prove Lemma \ref{L:compress}, it is enough to show that the orbit equivalence relation $E$
for the action of $G$ on $S$ is compressible. To see this,
for each $\ell \in \mathbb{N}$, let $S_{\ell}$ be the Borel subset of those $F \in S$
for which $\ell$ is the least integer such that if $x$, $y \in F$ and $x \mathbin{E_{0}} y$, 
then $x(n) = y(n)$ for all $n > \ell$. Clearly if $\ell < m$, then $S_{\ell} \preceq S_{m}$; and it follows that 
if $I \subseteq \mathbb{N}$ is an infinite co-infinite subset, then $A = \bigcup_{\ell \in I} S_{\ell}$ is a Borel
subset such that $S \sim A$ and $S \smallsetminus A$ is full.

\end{document}